\documentclass{amsart}
\usepackage{graphicx}
\usepackage{amssymb}
\usepackage{amsmath}
\usepackage{mathtools}

\newtheorem{thm}{Theorem}[section]
\newtheorem{lem}[thm]{Lemma}

\theoremstyle{definition}

\newcommand{\n}{\mathbf n}

\newcommand{\R}{\mathbf R}
\newcommand{\dist}{\operatorname{dist}}

\newcommand{\dv}{\operatorname{div}}

\makeatletter
\numberwithin{equation}{section}
\makeatother

\begin{document}

\title[Gradient maximum principle  for quasilinear parabolic equations]{On a gradient maximum principle  \\ for some quasilinear parabolic equations \\ on convex domains}
\author{Seonghak Kim}
       \address{Institute for Mathematical Sciences\\ Renmin University of China \\  Beijing 100872, PRC}
       \email{kimseo14@ruc.edu.cn}

\begin{abstract}
We establish a spatial gradient maximum principle for classical  solutions to the initial and Neumann boundary value problem of  some quasilinear parabolic equations on smooth convex  domains.
\end{abstract}
\maketitle

\section{Statement of main theorem}

In this note, we study the initial and Neumann boundary value problem of a quasilinear diffusion equation with a linear reaction term:
\begin{equation}\label{ibvp}
\begin{cases} u_t =\dv (\sigma(Du))-c(t)u &\mbox{in $\Omega\times (0,T]$},\\
{\partial u}/{\partial \n}=0 & \mbox{on $\partial \Omega\times (0,T]$},\\
u(x,0)=u_0(x) &\mbox{for $x\in \Omega$}.
\end{cases}
\end{equation}
Here, $\Omega\subset \R^n$  ($n\ge 1$)   is a bounded convex  domain with $C^{2}$ boundary, $T>0$ is any fixed number, $u=u(x,t)$ is the unknown function with $u_t$ and $Du=(u_{x_1},\cdots,u_{x_n})$ denoting its rate of change and spatial gradient respectively, $\n$ is the outer unit normal on $\partial\Omega$, $u_0\in C^{2}(\bar\Omega)$ is a  given initial function satisfying the compatibility condition:
\begin{equation}\label{ibvp-1}
{\partial u_0}/{\partial \n}=0\;\; \mbox{on $\partial \Omega,$}
\end{equation}
$c=c(t)\in W^{1,q_0}(0,T)$ is nonnegative for some $n+2<q_0<\infty$, and $\sigma\colon \R^n\to \R^n$ is given by $\sigma(p)=f(|p|^2)p$ ($p\in\R^n$) for some function $f\in C^3([0,\infty))$ fulfilling
\begin{equation*}
\lambda\le f(s)+2sf'(s)\le\Lambda\quad\forall s\ge 0,
\end{equation*}
where $\Lambda\ge \lambda >0$ are \emph{ellipticity} constants.  We easily have
\[
 \sigma^i_{p_j}(p)  = f(|p|^2)\delta_{ij} + 2f'(|p|^2) p_ip_j \quad (i,j=1,2,\cdots,n;\; p\in \R^n)
\]
and hence  the {\em uniform ellipticity} condition:
\begin{equation*}
 {\lambda} |q|^2 \le  \sum_{i,j=1}^n\sigma^i_{p_j}(p) q_iq_j\le \Lambda |q|^2\quad \forall\; p, \, q\in\R^n;
\end{equation*}
that is, (\ref{ibvp}) is a quasilinear uniformly parabolic problem  with conormal boundary condition.
Thus the existence, uniqueness and regularity of a classical solution $u$ to (\ref{ibvp}) follow from the standard theory such as in \cite[Theorem 13.24]{Ln} under suitable H\"older regularity assumptions on $u_0$ and $\partial\Omega$.

The main result of this note is  the following  theorem.

\begin{thm}[Gradient Maximum Principle]\label{thm:main}
If $u\in C^{2,1}(\bar\Omega_T)$ is a classical solution to problem (\ref{ibvp}), where $\Omega_T=\Omega\times(0,T]$, then it satisfies the \emph{gradient maximum principle:}
\begin{equation}\label{gmp-1}
\|Du\|_{L^\infty(\Omega_T)}=\|Du_0\|_{L^\infty(\Omega)}.
\end{equation}
\end{thm}


Gradient estimates for parabolic equations are usually given as  {\em a priori} estimates  depending on the initial datum, domain and ellipticity constants. Our result, Theorem \ref{thm:main}, gives an estimate independent of the {\em convex} domain and ellipticity constants. In case of the heat equation ($f\equiv 1$ and $c\equiv 0$),   (\ref{gmp-1}) was proved in \cite{Ka} for  $C^{3,1}$ solutions  and convex $C^3$ domains. Theorem \ref{thm:main} extends such a result to a large class of uniformly parabolic equations for $C^{2,1}$ solutions and convex $C^2$ domains. It is also important to note that the convexity assumption on the domain $\Omega$ in our result cannot be dropped in general; see a counterexample in \cite[Theorem 4.1]{AR}. Also, we refer the reader to \cite{PW, PS} for more extensive studies on the maximum principles in elliptic and parabolic differential equations.

Our motivation of (\ref{gmp-1}) is in the application of its pure diffusion case ($c\equiv 0$) to the study of the Neumann problem of some forward-backward diffusion equations \cite{KY, KY1, KY2}.
Although the proof of Theorem \ref{thm:main} would become  much easier if $u$ belonged to $C^{3,1}(\bar\Omega_T)$, the existence of such a solution $u$  often requires the initial datum $u_0$ lie in $C^{3+\alpha}(\bar\Omega)$ for some $0<\alpha<1$ and satisfy, in addition  to (\ref{ibvp-1}), the second compatibility condition:
\begin{equation}\label{ibvp-3}
\partial (\dv(\sigma(Du_0)))/\partial  \n=0 \;\; \mbox{ on $\partial\Omega.$}
\end{equation}
These requirements give rise to a subtle but critical issue  on the application of the convex integration method for constructing infinitely many  Lipschitz   solutions to certain  forward-backward parabolic Neumann problems.  For example, dealing with Perona-Malik type equations in \cite{KY}, condition (\ref{ibvp-3}) was posted for nonconstant radial initial data $u_0\in C^{3+\alpha}(\bar\Omega)$ when $\Omega$ is a ball. Also, an earlier version of the main existence  theorem in \cite{KY1} for the Perona-Malik equation  assumed that  initial data $u_0\in C^{3+\alpha}(\bar\Omega)$  with compatibility conditions (\ref{ibvp-1}) and (\ref{ibvp-3}) satisfy some technical restrictions, which cannot handle  the cases with  $\|Du_0\|_{L^\infty(\partial\Omega)}\ge 1$ or with $0<\|Du_0\|_{L^\infty(\partial\Omega)}< 1$ and $\|Du_0\|_{L^\infty(\Omega)}\ge \|Du_0\|_{L^\infty(\partial\Omega)}^{-1}.$   Our main result of this note  removes these requirements and restrictions on nonconstant  initial data $u_0$:  the only requirement  is that initial data $u_0\in C^{2+\alpha}(\bar\Omega)$ fulfill (\ref{ibvp-1}).

Another purpose of studying (\ref{gmp-1}) (when $c\equiv 0$) is to confirm the validity of \cite[Theorem 6.1]{KK} for \emph{convex} domains.  It has been a general belief that the initial-Neumann boundary value problem of a forward-backward parabolic equation in \cite{KK}  admits a unique global \emph{classical} solution if the initial datum $u_0\in C^{2+\alpha}(\bar\Omega)$ satisfies (\ref{ibvp-1}) and $\|Du_0\|_{L^\infty(\Omega)}<s_0$, where $s_0>0$ is the threshold at which the forward parabolicity of governing equation turns into the backward one. Regarding this, many authors often reported that such a problem is well-posed for \emph{subcritical} (or \emph{subsonic}) initial data. However, the  proof of \cite[Theorem 6.1]{KK} on such a result should be based on the gradient maximum principle (\ref{gmp-1}) for a modified uniformly parabolic problem of type (\ref{ibvp}), and so the convexity of the domain $\Omega$ should not be overlooked in the proof as pointed out above.

We finish this section with some comments on notations. We mainly follow the notations in the book \cite{LSU} for function spaces, with one exception  that the letter $C$ is used instead of $H$ regarding suitable (parabolic) H\"older spaces.
For integers $k,l\ge 0$ with $2l\le k$, we denote by $C^{k,l}(\bar\Omega_T)$ [resp. $C^{k,l}(\Omega_T)$] the space of functions $u\in C^0(\bar\Omega_T)$ [$u\in C^0(\Omega_T)$] such that $D^a_x D^j_t u\in C^0(\bar\Omega_T)$ [$D^a_x D^j_t u\in C^0(\Omega_T)$] for all multiindices $|a|\le k$ and integers $0\le j\le l$ with $|a|+2j\le k$.
We also adopt the summation convention that  repeated indices in a term represent the sum.

\section{Proof of main theorem}

We follow the notations and assumptions of Theorem \ref{thm:main}  and introduce two useful lemmas.
The convexity assumption on the domain $\Omega$ enters into the result (\ref{gmp-1}) through the following lemma from \cite[Lemma 2.1]{AR} or from \cite[Theorem 2]{Ka}; we do not reproduce the proof here.

\begin{lem}\label{lem-neumann}
Let $\Omega\subset\R^n$ be a bounded convex domain with $\partial\Omega$ of class $C^2$. If $w\in C^2(\bar\Omega)$ satisfies
$\partial w/\partial \n=0$ on $\partial\Omega,$
then $\partial(|Dw|^2)/\partial\n\le 0$ on $\partial\Omega.$
\end{lem}

The next lemma  gives an improved interior regularity of the solution $u\in C^{2,1}(\bar\Omega_T)$ to problem (\ref{ibvp}) that enables us to apply classical Hopf's Lemma for parabolic equations in a suitable setup. Its proof is postponed until the end of this section.

\begin{lem}\label{lem-ireg}
One has
\begin{equation}\label{lem-ireg-1}
u\in C^{3+\beta_0,\frac{3+\beta_0}{2}}(\Omega_T)
\end{equation}
for some $0<\beta_0<1$.
\end{lem}

We  now prove Theorem \ref{thm:main} based on the two lemmas above.

\begin{proof}[Proof of   Theorem \ref{thm:main}]

Let $v=|Du|^2$ on $\bar\Omega_T.$ By Lemma \ref{lem-ireg},   $v \in C^{1,0}(\bar\Omega_T)\cap C^{2,1}(\Omega_T).$ We compute, within $\Omega_T$,
\[
\Delta v =2 Du \cdot D(\Delta u) + 2 |D^2u|^2,
\]
\[
u_t= \dv (f(v)Du)-cu = f'(v) Dv \cdot Du + f(v) \Delta u- cu,
\]
\[
\begin{split}
Du_t =& f''(v) (Dv\cdot Du) Dv + f'(v) (D^2u)  Dv \\
&+ f'(v) (D^2v) Du + f'(v) (\Delta u) Dv + f(v) D(\Delta u) - c Du.\end{split}
\]
From these equations, using $v_t= 2Du\cdot Du_t$, we obtain
\begin{equation}\label{para-1}
v_t- \mathcal L (v) - V\cdot Dv =-2f(v) |D^2u|^2 - 2c |Du|^2\le 0 \;\;\mbox{in}\;\;\Omega_T,
\end{equation}
where $\mathcal L(v)$ and $V$ are defined by
\[
\mathcal L(v)= f(|Du|^2)\Delta v + 2 f'(|Du|^2) Du\cdot (D^2 v)Du,
\]
\[
V= 2f''(v) (Dv\cdot Du)Du   + 2 f'(v) (D^2u)  Du +
 2 f'(v) (\Delta u)  Du.
\]
Set  $\mathcal L (v)= a_{ij} v_{x_ix_j}$ with   coefficients $a_{ij}=a_{ij}(x,t)$, given by
\[
a_{ij}=\sigma^i_{p_j}(Du)=f(|Du|^2)\delta_{ij} + 2 f'(|Du|^2)u_{x_i}u_{x_j} \quad  (i,j=1,\cdots,n).
\]
Then, on $\bar\Omega_T$, all eigenvalues of the matrix $(a_{ij})$ lie in $[\lambda,\Lambda]$.

We now  show
\[
 \max_{(x,t)\in\bar\Omega_T}v(x,t) = \max_{x\in \bar\Omega} v(x,0),
\]
which completes the proof. We argue by contradiction; suppose
\begin{equation}\label{claim-1}
M:=\max_{(x,t)\in\bar\Omega_T}v(x,t) > \max_{x\in \bar\Omega} v(x,0).
\end{equation}
  Let  $(x_0,t_0)\in\bar\Omega_T$ with $v(x_0,t_0)=M;$ then $t_0>0.$ If $x_0\in \Omega$, then the strong maximum principle \cite{Ev} applied to (\ref{para-1})  would imply  that $v$ is constant on $\bar\Omega\times[0,t_0],$ which yields  $v(x,0)\equiv M$ on $\bar\Omega$, a contradiction to (\ref{claim-1}). Consequently, $x_0\in \partial\Omega$ and thus $v(x_0,t_0)=M>v(x,t)$ for all $(x,t)\in \Omega_T.$   We can then apply Hopf's Lemma for parabolic equations \cite{PW}  to (\ref{para-1}) to deduce $
 \partial v(x_0,t_0) /\partial \n >0,$
which contradicts  the conclusion of Lemma \ref{lem-neumann}.
 \end{proof}

We finally give the proof of Lemma \ref{lem-ireg}, although it may be well known to the experts in regularity theory.

\begin{proof}[Proof of  Lemma \ref{lem-ireg}]
We rely on \cite[Theorem III.12.1]{LSU} for the bootstrap of interior regularity for the solution $u\in C^{2,1}(\bar\Omega_T)$ to problem (\ref{ibvp}).
We divide the proof into several steps.

1. In $\Omega_T$,
\begin{equation}\label{ireg-1}
u_t=\dv (f(|Du|^2)Du)-cu =a_{ij}u_{x_i x_j}-cu,
\end{equation}
where $a_{ij}=\sigma^i_{p_j}(Du)=f(|Du|^2)\delta_{ij}+2f'(|Du|^2)u_{x_i}u_{x_j} \in C^{1,0}(\bar\Omega_T)$ and $c\in W^{1,q_0}(0,T)\subset C^{\alpha_0}([0,T])$ with $n+2<q_0<\infty$ and $\alpha_0:=1-1/q_0$. Note that the uniform ellipticity holds:
\begin{equation}\label{ireg-2}
 {\lambda}|\xi|^2\le a_{ij}(x,t)\xi_i\xi_j\le\Lambda|\xi|^2 \quad  \forall (x,t)\in\bar\Omega_T,\; \forall\xi\in\R^n.
\end{equation}

2. Fix an index $k\in\{1,\cdots,n\}$, and set $v=u_{x_k}\in C^{1,0}(\bar\Omega_T)$. Differentiating (\ref{ireg-1}) \emph{formally} with respect to $x_k$, we have
\begin{equation}\label{ireg-3}
v_t-\frac{\partial}{\partial x_j}(a_{ij}v_{x_i})+b_i v_{x_i}+cv =g,
\end{equation}
where
\begin{equation}\label{ireg-4}
b_i= (a_{ij})_{x_j},\; g= (a_{ij})_{x_k}u_{x_i x_j}\in C^{0}(\bar\Omega_T),\;\; c\in C^{\alpha_0}([0,T]).
\end{equation}
The membership (\ref{ireg-4}) easily verifies the admissible criteria (1.3)--(1.6)  in Chapter III of \cite{LSU} for coefficients and free term of equation (\ref{ireg-3}). It is also easy to see that $v\in  V^{1,0}_2(\Omega_T)$ is a weak (or generalized) solution to (\ref{ireg-3}) in the sense of \cite{LSU}.

To check some additional conditions in \cite[Theorem III.12.1]{LSU}, we rewrite equation (\ref{ireg-3}) in non-divergence form:
\begin{equation}\label{ireg-5}
v_t-a_{ij}v_{x_ix_j}+cv =g.
\end{equation}
Choose any $n+2<q<\infty$. From $a_{ij}\in C^{1,0}(\bar\Omega_T)$ and (\ref{ireg-4}), it follows that $a_{ij}$'s are bounded and continuous in $\Omega_T$, that $||c||_{L^q(\Omega\times(t,t+\tau))}\to 0$ as $\tau\to 0$ for each $t\in(0,T)$, and that $g\in L^q(\Omega_T)$; that is, coefficients and free term of equation (\ref{ireg-5}) fulfill the conditions in \cite[Theorem IV.9.1]{LSU} associated to the chosen number $q$.

With (\ref{ireg-2}), we can now apply \cite[Theorem III.12.1]{LSU} to obtain that weak derivatives $v_t,\,v_{x_i x_j}\,(i,j=1,\cdots,n)$ exist and belong to $L^q(Q)$ for all $1\le q<\infty$ and domains $Q\subset\Omega_T$ with $\dist(Q,\Gamma_T)>0$, where $\Gamma_T=\bar\Omega_T\setminus\Omega_T$ is the parabolic boundary of $\Omega_T$.

3. Fix any $\epsilon>0$ sufficiently small, and let
\[
\Omega^\epsilon=\{x\in\Omega\,:\,\dist(x,\partial\Omega)>\epsilon\},\;\;\Omega_T^\epsilon =\Omega^\epsilon\times(\epsilon,T].
\]
Also, fix any two indices $k,l\in\{1,\cdots,n\}$, and set $w=u_{x_k x_l}\in C^{0}(\bar\Omega_T^\epsilon)$. Then by  Step 2, we have $w\in V^{1,0}_2(\Omega_T^\epsilon)$.
Taking \emph{formal} derivative of (\ref{ireg-3}) in terms of $x_l$, we have
\begin{equation}\label{ireg-6}
w_t-\frac{\partial}{\partial x_j}(a_{ij} w_{x_i})+b_i w_{x_i}+cw =h,
\end{equation}
where
\[
h= (a_{ij})_{x_k x_l}u_{x_i x_j}+(a_{ij})_{x_k}u_{x_i x_j x_l}+(a_{ij})_{x_l}u_{x_i x_j x_k}.
\]
Since $f\in C^3([0,\infty))$, Step 2 implies
\begin{equation}\label{ireg-7}
h\in L^q(\Omega_T^\epsilon)\;\;\forall 1\le q<\infty.
\end{equation}
Observe that coefficients of equation (\ref{ireg-6}) are the same as those of equation (\ref{ireg-3}). Thus as in Step 2, with (\ref{ireg-7}), we see that the admissible criteria (1.3)--(1.6)  in Chapter III of \cite{LSU} are satisfied by coefficients and free term of equation (\ref{ireg-6}). Also, $w\in  V^{1,0}_2(\Omega_T^\epsilon)$ is a weak  solution to (\ref{ireg-6}).

As in Step 2, we also rewrite equation (\ref{ireg-6}) in non-divergence form:
\begin{equation}\label{ireg-8}
w_t-a_{ij}w_{x_ix_j}+cw=h.
\end{equation}
Likewise, coefficients of (\ref{ireg-8}) are equal to those of (\ref{ireg-5}), and free term $h$ satisfies (\ref{ireg-7}).

Again with (\ref{ireg-2}), it follows from \cite[Theorem III.12.1]{LSU} that weak derivatives $w_t,\,w_{x_i x_j}$ $(i,j=1,\cdots,n)$ exist and belong to $L^q(Q)$ for all $1\le q<\infty$ and domains $Q\subset\Omega_T^\epsilon$ with $\dist(Q,\Gamma_T^\epsilon)>0$, where $\Gamma_T^\epsilon=\bar\Omega_T^\epsilon\setminus\Omega_T^\epsilon$ is the parabolic boundary of $\Omega_T^\epsilon$.

4. Set $\tilde w=u_t\in C^0(\bar\Omega_T^\epsilon)$. By Step 2, we have $\tilde w\in V^{1,0}_2(\Omega_T^\epsilon)$.
Differentiating (\ref{ireg-1}) \emph{formally} with respect to $t$,
\begin{equation*}
\tilde w_t - \frac{\partial}{\partial x_j}(a_{ij}\tilde w_{x_i}) + b_i\tilde w_{x_i}+c\tilde w=\tilde h,
\end{equation*}
where
\[
\tilde h=(a_{ij})_t u_{x_i x_j}-c' u.
\]
From   Step 2 and $c'\in L^{q_0}(0,T)$, we have
\begin{equation*}
\tilde h \in L^{q_0}(\Omega_T^\epsilon).
\end{equation*}

As above, we obtain from \cite[Theorem III.12.1]{LSU} that $\tilde w_t=u_{tt}$ exists and belongs to $L^{q_0}(Q)$ for all domains $Q\subset\Omega_T^\epsilon$ with $\dist(Q,\Gamma_T^\epsilon)>0$.

5. By Steps 2--4, we   conclude that
\[
u\in W^{4,2}_{q_0}(\Omega_T^{2\epsilon})\;\;\forall\epsilon>0.
\]
By the parabolic Sobolev embedding theorem \cite[Lemma II.3.3]{LSU}, we obtain
\[
u\in C^{3+\beta_0,\frac{3+\beta_0}{2}}(\Omega_T),
\]
where $0<\beta_0<1-\frac{n+2}{q_0}$; hence (\ref{lem-ireg-1}) holds.
\end{proof}
\vspace{2ex}

{\bf Acknowledgments.}  The author would like to thank Professor Baisheng Yan and the referee for many helpful comments and suggestions to improve the presentation of the paper.

\end{document}